\documentclass[11pt]{amsart}
\usepackage[english]{babel}
\usepackage[utf8x]{inputenc}
\usepackage[T1]{fontenc}
\usepackage{a4wide}
\usepackage{mathpazo}
\usepackage[dvipsnames]{xcolor}
\usepackage[colorlinks=true,allcolors=MidnightBlue,linktoc=page]{hyperref}

\usepackage{mathrsfs}
\usepackage{amssymb, amsmath}
\usepackage{enumitem}
\usepackage{graphicx}
\usepackage{pgf,tikz}
\usepackage{mathrsfs}
\usetikzlibrary{arrows} 
\usepackage{xspace}
\usepackage[all]{xy}
\usepackage{bm}
\usepackage{scalerel,stackengine}
\stackMath
\newcommand\reallywidehat[1]{%
\savestack{\tmpbox}{\stretchto{%
  \scaleto{%
    \scalerel*[\widthof{\ensuremath{#1}}]{\kern-.6pt\bigwedge\kern-.6pt}%
    {\rule[-\textheight/2]{1ex}{\textheight}}
  }{\textheight}%
}{0.5ex}}%
\stackon[1pt]{#1}{\tmpbox}%
}
\newcommand{\calX}{\mathcal X}
\newcommand{\ov}{\overline}

\newcommand{\C}{\mathbf C}

\newcommand{\CC}{\mathbf{C}}
\newcommand{\NN}{\mathbf{N}}

\newcommand{\ZZ}{\mathbf{Z}}
\newcommand{\RR}{\mathbf{R}}
\newcommand{\R}{\mathbf{R}}

\newcommand{\KK}{\mathbf{K}}

\renewcommand{\H}{\mathbf{H}}
\newcommand{\K}{\mathbf{K}}

\newcommand{\cat}{{\upshape CAT($0$)}\xspace}

\DeclareMathOperator{\Isom}{Isom}
\DeclareMathOperator{\Aut}{Aut}

\DeclareMathOperator{\Span}{Span}

\DeclareMathOperator{\Conv}{Conv}

\DeclareMathOperator{\Id}{Id}

\DeclareMathOperator{\OO}{O}
\DeclareMathOperator{\PO}{PO}
\DeclareMathOperator{\PP}{P}
\DeclareMathOperator{\SO}{SO}
\DeclareMathOperator{\PSO}{PSO}
\DeclareMathOperator{\GL}{GL}
\DeclareMathOperator{\PGL}{PGL}
\DeclareMathOperator{\SL}{SL}
\newcommand{\PGO}{\mathrm{P}\Gamma\mathrm{O}}
\newcommand{\PGaL}{\mathrm{P}\Gamma\mathrm{L}}
\DeclareMathOperator{\Ofr}{O^{\mathfrak{fr}}_\KK}

\DeclareMathOperator{\PSOfr}{PSO^{\mathfrak{fr}}_\KK}
\newcommand{\X}{\mathcal{X}}
\newcommand{\Y}{\mathcal{Y}}
\theoremstyle{plain}
\newtheorem{thm}{Theorem}[section]
\newtheorem*{thm*}{Theorem}
\newtheorem{lem}[thm]{Lemma}
\newtheorem{prop}[thm]{Proposition}
\newtheorem{cor}[thm]{Corollary}
\theoremstyle{definition}
\newtheorem*{defn*}{Definition}
\newtheorem{defn}[thm]{Definition}

\newtheorem*{example*}{Example}
\newtheorem{rem}[thm]{Remark}

\newtheorem*{rem*}{Remark}
\begin{document}
\title[Representations of infinite dimensional orthogonal groups with finite index]
{Representations of infinite dimension orthogonal groups of quadratic forms with finite index}
\author[B. Duchesne]{Bruno Duchesne}
\address{Institut \'Elie Cartan, UMR 7502, Universit\'e de Lorraine et CNRS, Nancy, France.}
\date{July 2019}
\begin{abstract}We study representations $G\to H$ where $G$ is either a simple Lie group with real rank at least 2 or an infinite dimensional orthogonal group of some quadratic form of finite index at least 2 and $H$ is such an orthogonal group as well. The real, complex and quaternionic cases are considered. Contrarily to the rank one case, we show that there is no exotic such representations and we classify these representations. 

On the way, we make a detour and prove that the projective orthogonal groups $\PO_\K(p,\infty)$ or their orthochronous component  (where $\K$ denotes the real, complex or quaternionic numbers) are Polish groups that are topologically simple but not abstractly simple.
\end{abstract}
\thanks{This works has been stimulated by discussions with Nicolas Monod and Pierre Py. I thank them for these stimulations and their comments about this paper. The author is partially supported  by projects ANR-14-CE25-0004 GAMME and ANR-16-CE40-0022-01 AGIRA. A part of this work was done during the semester "Geometric and Analytic Group Theory" at the Institute of Mathematics of the Polish Academy of Sciences and thus was partially supported by the grant 346300 for IMPAN from the Simons Foundation and the matching 2015-2019 Polish MNiSW fund.}
\maketitle
\section{Introduction}
The study of finite dimensional representations of Lie groups is a classical subject. Apart from finite dimensional representations, there are also infinite dimensional unitary representations, which are classically studied. In this paper, we are interested in some other infinite dimensional representations with a geometric taste. Namely, representations that preserve a quadratic or Hermitian form with finite index.

The simplest example is given by representations into the real orthogonal group $\OO(1,\infty)$ which is the group that preserves a quadratic form of signature $(1,\infty)$ on some separable real Hilbert space. In this case, the geometric taste is given by the induced action on the infinite dimensional hyperbolic space. Such representations were put in this geometric context in \cite{MR2881312,MR3263898}. The Cremona group has also a natural representation in $\OO(1,\infty)$ see for example \cite{MR2811600}.

In \cite{MR3263898}, representations $\PO(1,n)\to\PO(1,\infty)$ are studied and classified by a parameter $t\in(0,1]$. These representations correspond to a standard embedding if and only if $t=1$. In the other cases, these representations are called \emph{exotic}.  They come from the spherical principal series of $\PO(1,n)$. This spherical principal series also yields representations $\PO(1,n)\to\PO(p,\infty)$ for infinitely many $p>1$, where the possible values of $p$ depend on $n$. 

This result has been extended to the classification of self-representations of $\PO(1,\infty)\to\PO(1,\infty)$ in \cite{AHL_2019__2__259_0} and there is still a one parameter family of exotic self-representations.\\

Here we are interested in the higher rank cases (for the source group) and in particular in the existence of possible higher rank exotic representations. For uniform lattices in higher rank semisimple Lie groups, a theorem similar to the geometric interpretation of Margulis superrigidity has been proved in \cite[Theorem 1.2]{MR3343349}. This gives hints that there should be no exotic such representations for higher rank semisimple Lie groups.

Representations in $\PO(p,\infty)$, for $p\in\NN$, are not the only ones to give actions on infinite dimensional and finite rank symmetric spaces with non-positive curvature. One could also consider the similar constructions over the complex numbers and the quaternions. This gives rise to representations in $\PO_\K(p,\infty)$ where $\K=\R,\C$ or $\H$ and the associated symmetric spaces are denoted $\calX_\K(p,\infty)$ \cite{MR3044451,BD15} (see Sections 2 and 3 for precise definitions). We use the notations $\PO_\K(p,\infty)$  to have a uniform notation independent of the ground field (or division algebra) and $\PO(p,\infty)$ merely means $\PO_\R(p,\infty)$. Self-representations of $\PO_\C(1,\infty)$ have been also classified in \cite{monod2018notes} under an additional hypothesis.

The study of representations of these infinite dimensional classical groups is not completely new and unitary representations have been studied for example in \cite{ol1978unitary} and references therein. Even earlier, these infinite dimensional groups were studied by Pontryagin, Naimark and Ismagilov, see for example  \cite{MR0195991, MR0201569,Sasvari1990}.\\

Our first result is about finite dimensional Lie groups of higher rank. It shows that there is no exotic continuous representations in this case. All those representations come from finite dimensional representations and unitary representations.

 \begin{thm}\label{splitting}Let $G$ be a connected simple non-compact Lie group with trivial center. Let $G\to\PO_\K(p,\infty)$ be a continuous representation without totally isotropic invariant subspace.\\
 If the real rank of $G$ is at least 2 then the underlying Hilbert space $\mathcal{H}$ splits orthogonally as $E_1\oplus\dots\oplus E_k\oplus \mathcal{K}$ where each $E_i$ is a finite dimensional, non-degenerate, $G$-invariant linear subspace and the induced representation on $E_i$ is irreducible. The induced representation on $\mathcal{K}$ is unitary.
 \end{thm} 

Let us observe that the existence of a decomposition as a sum of finitely many irreducible representations and  a unitary one is known  for any group as soon as there is no totally isotropic invariant space \cite{MR0201569,Sasvari1990}. Moreover, this theorem extends the results of \cite{MR0195991} where it is proved for $\SL_2(\C)$. The strategy to prove Theorem~\ref{splitting} is to use the aforementioned mention \cite[Theorem 1.2]{MR3343349} for lattices and extended it to the whole ambiant group. This strategy may seem surprising since the ambiant Lie group has much more structure than its lattices. In particular, it has a differentiable structure. The topology used on $\PO_\K(p,\infty)$ is the coarsest that makes the action on the symmetric space $\calX_\K(p,\infty)$ continuous. This is not the topology coming from the norm topology on $\PO_\K(p,\infty)$ and thus one cannot use the standard result in finite dimension that a continuous homomorphism between Lie groups is actually smooth. In fact, the exotic representations $\PO(1,n)\to\PO(1,\infty)$ cited above are not continuous for the norm topology.

One can also imagine that one can directly adapt the proof for lattices to Lie groups but the proof for lattices uses harmonic maps from  a locally symmetric space. For the whole Lie groups this would lead to harmonic maps from a point. Over a point, differential methods could not work. 

Theorem~\ref{splitting} leaves open the study of representations of rank one Lie groups in the groups $\PO_\K(p,\infty)$ (See \cite[Problem 5.2]{MR3263898}). Let us mention that continuous representations of $\PSO(1,n)$ for $n$ finite into $\OO(2,\infty)$ have been considered in \cite{py2018hyperbolic} and the authors show there is no such irreducible representations. 

\begin{thm}\label{quad} Let $p,q\in \NN$ with $p\geq2$ and $\K,\mathbf{L}\in\{\R,\C,\H\}$, If $\rho\colon \PO_\K(p,\infty)\to\PO_{\mathbf{L}}(q,\infty)$ is a continuous geometrically dense representation then $\K=\mathbf{L}$, $q=p$ and $\rho$ is induced by an isomorphism of quadratic spaces.  \end{thm}

The assumption about geometric density implies that the representation is irreducible and in finite dimension this is equivalent to Zariski density. See \cite[\S3]{duchesne2018boundary} for a discussion thereon. As a corollary, we get an understanding of the continuous automorphism group.

\begin{cor}The group of continuous automorphisms  $\Aut_c(\PO_\K(p,\infty))$ is $\Isom\left(\calX_\K(p,\infty)\right)$.\end{cor}

This latter group is described in Theorem~\ref{isomgroup} where it is shown that $\Isom\left(\calX_\K(p,\infty)\right)$ is $\PO_\K(p,\infty)$ except in the complex case where the only non-trivial outer automorphism correspond to the complex conjugation. 

Since the topology on $\PO_\K(p,\infty)$ plays a role in the previous results, we take a quick look at this topological group and prove that, although abstractly it is not  simple, topologically it is.

 \begin{thm}\label{thmtopsimp} The Polish group $\PO_\C(p,\infty),\PO_\H(p,\infty)$ and $\PO_\R^o(p,\infty)$ are topologically simple.
 \end{thm}
 Here $\OO^o_\R(p,\infty)$ denotes the orthochronous component of $\OO_\R(p,\infty)$. See \S~\ref{sec_top_simp}.
 
  \begin{rem} All the results in this introduction are stated for the infinite countable cardinal $\aleph_0$, merely denoted by $\infty$ here. But the proofs deal with any cardinal $\kappa$ except for the fact that $\PO_\K(p,\kappa)$ is Polish, which holds only when $\kappa$ is countable. Theorems are restated in this larger generality in the body of the text.
 \end{rem}
 
\begin{rem} It is asked in \cite[\S1.3]{AHL_2019__2__259_0} if the group $\PO(1,\infty)$ has the automatic continuity. A Polish group $G$ has this property if any homomorphism to a separable topological group is continuous. The same question can be asked for $\PO_\K(p,\infty)$ and any finite value of $p$. If this property holds for $\PO_\K(p,\infty)$ then the continuity assumption in Theorem~\ref{quad} and its corollary can be removed.
A general study of $\PO_\K(p,\infty)$ as a Polish group should be the subject of a future work.
\end{rem}

\tableofcontents

\section{Symmetric spaces of infinite dimension}

By a Riemannian manifold, we mean a (possibly infinite dimensional) smooth manifold modeled on some real Hilbert space with a smooth Riemannian metric. For a background on infinite dimensional Riemannian manifolds, we refer to \cite{MR1666820} or \cite{MR2243772}. 

Let $(M,g)$ be a Riemannian manifold, a \emph{symmetry} at a point $p\in M$ is an involutive isometry  $\sigma_p\colon M\to M$ such that $\sigma_p(p) = p$ and the differential at $p$ is $-\Id$. A \emph{Riemannian symmetric space} (or simply a symmetric space) is a connected Riemannian manifold such that, at each point, there exists a symmetry. See \cite[\S3]{BD15} for more details.

A natural infinite dimensional generalization of the symmetric space associated to  $\SL_n(\RR)$ is obtained with the following construction. Let $\mathcal{H}$ be a real Hilbert space and $L^2(\mathcal{H})$ be the space of Hilbert-Schmidt operators on $\mathcal{H}$. Let $S^2(\mathcal{H})$ be the subset of $L^2(\mathcal{H})$ given by self-adjoint operators and let $P^2(\mathcal{H})=\exp\left(S^2(\mathcal{H})\right)$, the set of Hilbert-Schmidt perturbations of the identity that are positive definite. This space $S^2(\mathcal{H})$ is a symmetric space of non-positive curvature.

As in finite dimension, the Riemann tensor and the curvature operator can be defined for Riemannian manifolds of infinite dimension. Under the assumption of separability and non-positivity of the curvature operator, a classification of such symmetric spaces has been obtained in \cite[Theorem 1.8]{BD15} and they are analogs of the classical ones. 

The non-positivity of the curvature operator is a natural but stronger condition than non-positivity of the sectional curvature. In particular, symmetric spaces with non-positive sectional curvature and no local de Rham factor are automatically simply connected and thus \cat \cite[Proposition 4.1]{BD15}.

\begin{defn} Let $\mathcal{X}$ be a Riemannian manifold. A submanifold $\mathcal{Y}$ is said to be \emph{totally geodesic} if  for any geodesic $\gamma\colon I\to X$, where $I$ is a real open interval containing $0$, with initial conditions $\left(\gamma(0),\gamma'(0)\right)$ in the tangent bundle of $\mathcal{Y}$, $\gamma(I)$ is contained in $\mathcal{Y}$.
\end{defn}

The \emph{rank} of a symmetric space with non-positive sectional curvature is defined as the supremum of the dimensions of totally geodesic embedded Euclidean spaces. The symmetric space $P^2(\mathcal{H})$ has infinite rank and for example, one can find three points that are not contained in any finite dimensional totally geodesic subspace \cite[Example 2.6]{duchesne2018boundary}.

One can also construct  infinite dimensional symmetric spaces of non-positive curvature and finite rank. Let us describe them. Let $\H$ denotes the division algebra of the quaternions, and $\mathcal {H}$ be a Hilbert space over $\K=\R$, $\C$ or $\H$ with a Hilbert basis of cardinality $\kappa$. In case $\K=\H$,  the scalar multiplication is understood to be on the right. Let $p\in\NN$. We fix an orthonormal basis $(e_i)_{i\in p+\kappa}$ of the separable Hilbert space $\mathcal{H}$, and we consider the quadratic form 
$$Q(x)=\sum_{i\leq p}\ov x_ix_i-\sum_{i\in \kappa}\ov x_ix_i$$ where  $x=\sum e_ix_i$. The space 
 
$$\calX_{\mathbf K}(p,\kappa)=\left\{V\leq\mathcal{H},\ \dim_\K(V)=p,\ Q|_V>0\right\}$$  

has a structure of symmetric space of non-curvature (see \cite{MR3044451}). The rank of $\calX_{\mathbf K}(p,\kappa)$ is exactly $p$. Actually, separable symmetric spaces of non-positive curvature operator and finite rank are classified \cite[Corollary 1.10]{BD15}. They split as a finite product of finite dimensional symmetric spaces of non-compact type and copies of $\calX_{\mathbf K}(p,\infty)$. These infinite dimensional symmetric spaces have the particularity that any finite configurations of points, flats subspaces or points at infinity are contained in a finite dimensional totally geodesic subspace \cite[Proposition 2.6]{MR3044451}.\\ 

The following characterizations of totally geodesic subspaces of symmetric spaces are well-known in finite dimension (it follows from the work of Mostow \cite{MR0069829}). In particular, one can characterize them without referring to the differential structure. The knowledge of geodesics or the symmetries is sufficient. 

\begin{prop}\label{theo}Let $\mathcal{X}$ be a separable symmetric space with non-positive operator curvature and finite rank or $\mathcal{X}=\calX_\K(p,\kappa)$. Let $\mathcal{Y}$ be a non-empty closed subspace of $\mathcal{X}$. Let $\sigma_x$ be the symmetry at $x\in \mathcal{X}$.The followings are equivalent:
\begin{enumerate}
\item The subspace $\mathcal{Y}$ is a totally geodesic submanifold of $\mathcal{X}$.
\item For any distinct points $x,y\in \mathcal{Y}$, the unique geodesic of $\mathcal{X}$ containing $x$ and $y$ is included in $\mathcal{Y}$.
\item The subspace $\mathcal{Y}$ is convex and satisfies $\sigma_x(\mathcal{Y})=\mathcal{Y}$ for any $x\in \mathcal{Y}$.
\item The subspace $\mathcal{Y}$ is a connected submanifold such that the Riemannian distance of $\mathcal{Y}$ (coming from the Riemannian structure of $\mathcal{Y}$ induced by $\mathcal{X}$) is equal to the restriction to $\mathcal{Y}$ of the Riemannian distance of $\mathcal{X}$.
\item There exists $x\in \mathcal{Y}$ such that $\mathcal{Y}=\exp_x(E)$ for some closed linear subspace  $E\leq T_x\mathcal{X}$ such that $E$ is a Lie triple system when $T_x\mathcal{X}$ is identified to $\mathfrak{p}$ in the Cartan decomposition $\mathfrak{g}=\mathfrak{t}\oplus\mathfrak{p}$ relative to $x$.
\item For any $x\in \mathcal{Y}$,  $\mathcal{Y}=\exp_x(E)$ for some closed linear subspace  $E\leq T_x\mathcal{X}$ such that $E$ is a Lie triple system when $T_x\mathcal{X}$ is identified to $\mathfrak{p}$ in the Cartan decomposition $\mathfrak{g}=\mathfrak{t}\oplus\mathfrak{p}$ relative to $x$.
\end{enumerate}
\end{prop}

\begin{proof} Since $\mathcal{X}$ is a symmetric space with non-positive curvature operator, we know that it can be embedded as a totally geodesic manifold of $S^2(\mathcal{H})$ \cite{BD15}. Thus, it suffices to consider the case where $\mathcal{X}$ is $S^2(\mathcal{H})$. In that case, the correspondence between totally geodesic submanifolds and Lie triple system has been observed in \cite[Proposition III.4]{MR0476820}. So we know that (1), (5) and (6) are equivalent.

Let us prove that (2) and (3) are equivalent. It is clear that if $\mathcal{Y}$ satisfies (2) then it is convex. Moreover, for any $x\neq y$, $\sigma_x(y)$ lies on the geodesic line through $x$ and $y$. Thus $\mathcal{Y}$ satisfies (3). Conversely if $\mathcal{Y}$ satisfies (3) then for any $x\neq y\in\mathcal{Y}$, the geodesic segments $\left[\left(\sigma_x\circ\sigma_y\right)^n(y), \left(\sigma_y\circ\sigma_x\right)^n(x)\right]$ is contained in $\mathcal{Y}$ for all $n\in\ZZ$. Since their union is the whole geodesic through $x$ and $y$, this geodesic is contained in $\mathcal{Y}$.

Assume that $\mathcal{Y}$ satisfies (6) then for any $x\neq y$ in $\mathcal{Y}$, there is $v\in F$ such that $y=\exp_x(v)$. Since the geodesic through $x$ and $y$ is the image of $t\mapsto \exp_x(tv)$ for $t\in\RR$, this geodesic is contained in $\mathcal{Y}$ and thus $\mathcal{Y}$ satisfies (2). 

Now assume $\mathcal{Y}$ satisfies (2). For a finite subset $F\subset \mathcal{Y}$, $F$ lies in some finite dimensional totally geodesic subspace $\mathcal{Z}_F$. The intersection $\mathcal{Y}_F=\mathcal{Y}\cap\mathcal{Z}_F$ is a totally geodesic subspace of $\mathcal{Z}_F$ and thus there is a finite dimensional subspace $E_F$ of $T_{y_1}\mathcal{X}$, that is a Lie triple system (by the result in finite dimension) and such that $\exp_{y_1}(E_F)=\mathcal{Y}_F$. Since $\exp_{y_1}\colon T_{y_1}\mathcal{X}\to\mathcal{X}$ is a homeomorphism, it induces a homeomorphism from $E$, the closure of $\cup_F E_F$ (where $F$ is any finite subset of $\mathcal{Y}$) to $\mathcal{Y}$ and thus $\mathcal{Y}$ satisfies (5).

It is clear that (2) implies (4). Conversely, since the distance on $\mathcal{Y}$ coincides with the one in $\mathcal{X}$, $\mathcal{Y}$ is convex. Moreover since $\mathcal{Y}$ is a submanifold, any geodesic segment can be enlarged a bit and since $\calX$ is closed, any geodesic segment can be extended to a whole geodesic line. Thus (2) is satisfied.
\end{proof}
The following corollary is an immediate consequence of characterization (1).

\begin{cor}\label{cor_min_tot} Any non-empty intersection of totally geodesic subspaces of $\mathcal{X}$ is a totally geodesic subspace. In particular, any subset of $\mathcal{X}$ is contained in a unique minimal totally geodesic subspace of $\mathcal{X}$.

\end{cor}
\section{Isometry groups of infinite dimensional Riemannian symmetric spaces of finite rank}

\subsection{Full isometry group} In the remaining of this paper, $\K$ denotes either the real, complex or quaternionic numbers. Let $\mathcal{H}$ be a $\KK$-Hilbert space of with a strongly non-degenerate quadratic form $Q$ of signature $(p,\kappa)$ where $p\in\NN$ and $\kappa\geq p$ is some finite or infinite cardinal. This means that 
\begin{itemize}
\item there is a $Q$-orthogonal decomposition $\mathcal{H}=\mathcal{H}_+\oplus\mathcal{H}_-$, 
\item if $\varphi$ is the Hermitian form  obtained by polarization of $Q$ and $\varphi_\pm$ are the restrictions of $\varphi$ on $\mathcal{H}_\pm$ then $\left(\mathcal{H}_+,\varphi_+\right)$ and $\left(\mathcal{H}_-,-\varphi_-\right)$ are Hilbert spaces with  Hilbert bases of cardinality respectively $p$ and $\kappa$,
\item moreover, the Hermitian form $\varphi_+-\varphi_-$ is positive definite on $\mathcal{H}$  and equivalent to the scalar product on $\mathcal{H}$.
\end{itemize}
For more details about strongly non-degenerate quadratic forms, one may have a look at \cite[\S2]{MR2152540}.
We are essentially interested in the case where $\kappa$ is infinite but at least we assume that $p+\kappa\geq4$. If $\kappa$ is the infinite countable cardinal, we denote it by $\kappa=\infty$ as in the introduction.
 
We denote the orthogonal group of $Q$ by $\OO_\KK(p,\kappa)$ and its intersection with the set of finite rank perturbations of the identity by $\Ofr(p,\kappa)$.  Let us recall that a finite rank perturbation of the identity is an operator of the form $\Id+ A$ where $A$ is a finite rank operator. The center of $\OO_\KK(p,\kappa)$ is the set of homotheties $\lambda \Id$ where $\lambda\in \K$, $|\lambda|=1$ and $\lambda\in\mathcal{Z}(\K)$, the center of $\K$. The center of $\OO_\KK(p,\kappa)$ is thus $\pm\Id$ for $\K=\RR$ or $\H$ and isomorphic to $\mathbf{S}^1$ for $\K=\CC$. We define $\PO_\K(p,\kappa)$ to be the quotient of $\OO_\K(p,\kappa)$ by the its center $\mathcal{Z}\left(\OO_\K(p,\kappa)\right)$. 

By construction $\PO_\K(p,\kappa)$ acts by isometries on $\calX_\KK(p,\kappa)$ and it is proved that when $\KK=\RR$, $\PO(p,\kappa)=\Isom\left(\calX_\KK(p,\kappa)\right)$. As explained in \cite{monod2018notes} for $p=1$, $\Isom\left(\calX_\CC(1,\kappa)\right)$ is the union of the holomorphic isometries and the antiholomorphic isometries. This comes from the complex conjugation which is the unique field automorphism of $\CC$ that preserves the absolute value. For the quaternionic numbers, one can  follows the strategy as in finite dimension to prove that $\Isom(\calX_\H(1,\kappa))=\PO_\K(1,\kappa)$, see \cite[II.10.17-21]{MR1744486}. 

Projective geometry is lurking in these statements and the fundamental theorem of projective geometry tells us that field automorphisms have some role to play. As it is well known, the field $\RR$ has no non-trivial field automorphisms and for the quaternions $\H$, all field automorphisms are inner and preserve the quaternionic absolute value. Since the center of $\H$ is $\RR$, one has $\Aut(\H)\simeq \H^\ast/\RR^\ast\simeq\SO(3)$. The former isomorphism is obtained by considering the two dimensional sphere of pure unit quaternions. Let us observe that the quaternionic conjugation is not a field automorphism since $\overline{xy}=\overline{y}\, \overline{x}$  which is in general different from $\overline{x}\,\overline{y}$.

Let us denote by $\Aut_c(\K)$ the group of continuous field automorphisms of $\K$ . It coincides with the group of automorphisms that preserve the absolute value ($\left|\sigma(x)\right|^2=\left|x\right|^2$ for $x\in\K$ and $\sigma\in\Aut(\K)$). It is trivial when $\K=\RR$, $\Aut_c(\CC)\simeq\ZZ/2\ZZ$ (the non-trivial automorphism being given by the conjugation) and $\Aut_c(\H)=\Aut(\H)\simeq\H^*/\RR^*\simeq\SO(3)$ is given by conjugations of elements of $\H^*$.
One can realize 
 each $\sigma\in\Aut_c(\K)$ as a $\R$-linear isomorphism of $\mathcal{H}$ that preserves the real part of the Hermitian form. It suffices to apply coordinate-wise the automorphism $\sigma$ in the base $(e_i)$. In particular, it maps $\K$-linear subspaces to $\K$-linear subspaces of the same dimension and preserves positivity. Thus it induces a bijection of $\calX_\K(p,\kappa)$ and the metric is invariant. Observe that this construction identifies $\Aut_c(\sigma)$ with a subgroup of the stabilizer of the span of $(e_1,\dots,e_p)$ in $\calX_\K(p,\kappa)$.
 
 \begin{rem}\label{rem_embedding} A possibility to see that each $\sigma\in\Aut_c(\K)$ induces an isometry is the following. Let $d=\dim_\R(\K)$ and let $\varphi_\RR=\Re(\varphi)$ be the real part of the Hermitian form $\varphi$ on the underlying real Hilbert space structure on $\mathcal{H}$. The real quadratic form $Q_\RR$ is a strongly non-degenerate quadratic form of signature $(dp,d\kappa)$ (where $d\kappa$ denotes the multiplication of cardinal numbers and thus $d\kappa=\kappa$ as soon as $\kappa$ is infinite since $d\leq4$). Considering any totally isotropic $\K$-linear space of dimension $p$ as a  totally isotropic $\RR$-linear space of dimension $dp$ gives an embedding of $\calX_\K(p,\kappa)$ as a totally geodesic submanifold of $\calX_\R(dp,d\kappa)$ and any element of $\OO_\K(p,\kappa)$ or of $\Aut_c(\K)$ act on $\calX_\R(dp,d\kappa)$ preserving both $\calX_\K(p,\kappa)$ and $\varphi_\RR$. Thus any such element induces an isometry of  $\calX_\K(p,\kappa)$.
 \end{rem}
 
 Our goal in this subsection is to identify the full isometry group of $\calX_\K(p,\kappa)$. So we need to understand how elements of $\OO_\K(p,\kappa)$ and $\Aut_c(\K)$ interact in $\Isom\left(\calX_\K(p,\kappa)\right)$ and see how they generate it. Let us recall a few vocabulary from projective geometry. A \emph{collineation} of the projective space $\PP\mathcal{H}$ is a bijection that preserves projective lines (they are also called projective automorphisms). We denote by $\PGaL(\mathcal{H})$ the group of collineations. Since the dimension of $\mathcal{H}$ is at least 3, the fundamental theorem of projective geometry tells us that this group is the image of the group of semilinear automorphisms $\GL(\mathcal{H})\rtimes\Aut(\K)$. Let us denote by $\PGO_\K(p,\kappa)$ the image of $\OO_\K(p,\kappa)\rtimes\Aut_c(\K)$ in  $\PGaL(\mathcal{H})$. Since $\R$ has no automorphisms, $\PGO_\R(p,\kappa)=\PO_\R(p,\kappa)$, for  the complex numbers $\PGO_\CC(p,\kappa)=\PO_\CC(p,\kappa)\rtimes \ZZ/2\ZZ$ and remarkably for the quaternions  $\PGO_\H(p,\kappa)=\PO_\H(p,\kappa)$ since the conjugation by $a\in\H^*$ induces the same collineation as the left multiplication by  $a$.
 
 In order to identify the isometry group $\Isom\left(\calX_\K(p,\kappa)\right)$, we introduce the context of Tits' fundamental work on spherical buildings \cite{MR0470099}. Let us recall briefly that a spherical building of dimension $n$ is a simplicial  complex with some special subcomplexs called appartements  and isomorphic to a tessellation of a real sphere of dimension $n$ associated to some spherical Coxeter group. It is proved in \cite[Proposition 5.2]{MR3044451} that the Tits boundary $\partial \calX_\K(p,\kappa)$ is a spherical building of dimension $p-1$ (in particular, this is a non-trivial structure as soon as $p\geq2$). Actually the proof is done in the case where $\kappa=\infty$ but the proof works for any cardinal as well.
 
 One can describe explicitly what are the simplices. They are in correspondance with isotropic flags, that are sequences $F_1<F_2<\dots<F_k$ where each $F_i$ is a  totally isotropic $\K$-linear subspace (and thus $k\leq p$). This  is proved in the real case in \cite[Proposition 6.1]{MR3044451}. This can be extended to the other cases via the embedding described in Remark \ref{rem_embedding} since an element of $\OO_\K(p,\kappa)$ that stabilizes a flag of isotropic $\RR$-linear subspaces for $Q_\RR$ also stabilizes a flag of $\K$-linear subspaces and vice-versa.
 
 In particular, vertices correspond to totally isotropic subspaces. The \emph{type} of a vertex is merely the dimension of the associated totally isotropic subspace. The \emph{polar space} $\mathcal{S}$ associated to this building (see \cite[\S7]{MR0470099} for general definitions) is the space of isotropic lines, i.e. vertices of type 1. Two points in $\mathcal{S}$ are \emph{collinear} if they are contained in a common totally isotropic plane, i.e. they are orthogonal. An \emph{automorphism} of this polar space is a bijection that preserves collinearity. 
 
 To any $\xi\in\partial \calX_\K(p,\kappa)$, one can associate its symmetric space at infinity $\calX_\xi$. In a general \cat context, this is defined as the quotient of the space of geodesics pointing to $\xi$ under the equivalence of being strongly asymptotic (see for example \cite[\S4.2]{MR2495801}). Two points $\xi,\eta\in\partial \calX_\K(p,\kappa)$ are \emph{opposite} if they are extremities of a common geodesic. If $\xi$ is a vertex and $\eta$ is opposite to $\xi$ then $\eta$ is a vertex of the same type and the span of $\xi$ and $\eta$ is a linear subspace of $\mathcal{H}$ of signature $(d,d)$ where $d$ is the type of $\xi$.  This follows from the fact that any geodesic line lies in some maximal flat subspace and from the description of these flats in \cite[\S3]{MR3044451}.
 
 Let us fix a vertex $\xi\in\calX_\K(p,\kappa)$ of type $d$. Let us choose $\eta$ opposite to $\xi$. The union $\Y$ of geodesics  with extremities $\xi,\eta$ is closed and invariant under symmetries $\sigma_x$ for $x\in\Y$, thus it is a totally geodesic subspace of $\calX_\K(p,\kappa)$. It splits as $\Y\simeq \R\times \calX_\xi$ where the $\R$-factor correspond the direction of a geodesic from $\xi$ to $\eta$. The isometry type of $\calX_\xi$ does not depend on $\eta$ since the stabilizer of $\xi$ in $\PO_\K(p,\kappa)$ acts transitively on points opposite to $\xi$. A point $x\in\X_\K(p,\kappa)$, i.e. a positive subspace of dimension $p$ lies in $\Y$ if and only if $x\cap\Span\left\{\xi,\eta\right\}$ has dimension $d$ and $x=x\cap\Span\left\{\xi,\eta\right\}\oplus x\cap\Span\left\{\xi,\eta\right\}^\bot$. Since $\Span\left\{\xi,\eta\right\}^\bot$ has signature $(p-d,\kappa)$, $\calX_\xi$ splits has $\calX_\xi\simeq\calX'_\xi\times \calX_\K(p-d,\kappa-d)$ where $\calX'_\xi$ is some finite dimensional symmetric space of non-compact type.
 
 For the next lemma, let us assume that $\kappa$ is an infinite cardinal, otherwise an identification (which is not very hard) of $\calX'_\xi$ is required.
 
 \begin{lem}\label{polar_aut} Let $p\geq2$. The group $\Isom\left(\calX_\KK(p,\kappa)\right)$ acts by type preserving automorphisms on the spherical building $\partial \calX_\K(p,\kappa)$.
 
 Moreover, $\Isom\left(\calX_\KK(p,\kappa)\right)$ acts by automorphisms on the polar space $\mathcal{S}$.
 \end{lem}
 
 \begin{proof}Isometries preserve maximal flats and thus their boundaries which are appartements of the spherical building at infinity. Points in the interior of chambers (maximal simplices) correspond to extremities of regular geodesics, those contained in a unique maximal flat. Thus the set of these points is invariant by the isometries as well as chambers, which are connected components of the set of regular points at infinity. The other simplices are obtained as intersections of closure of chambers and thus invariant as well. In particular, vertices are minimal non-empty intersections of closure of chambers and thus the isometry group acts on the set of vertices.
 
 We claim that any two vertices have the same type if and only if they are in the same $\Isom\left(\calX_\KK(p,\kappa)\right)$-orbit. Witt theorem implies that $\PO_\K(p,\kappa)$ acts transitively on the set of vertices of a given type. For a vertex $\xi$ of type $d$, the infinite dimensional factor of $\X_\xi$ has rank $p-d$ and this rank is invariant by $\Isom\left(\calX_\KK(p,\kappa)\right)$.
 
 So, this proves that the action on the spherical building at infinity is by type-preserving automorphisms. Looking at vertices of type 1, we get an action on the polar space $\mathcal{S}$. For two distinct vertices of type 1, there are only two possibilities: they are collinear, that is their span is totally isotropic, that is they are orthogonal or their span has signature $(1,1)$. This can be recovered by geometric means, since in the first case, their Tits angle is $\pi/2$ where as they are opposite in the second case.
 
 So the action of the isometry group on the polar space $\mathcal{S}$ preserves collinearity.
 \end{proof}
 
\begin{thm}\label{isomgroup} Let $p\in\NN$. The isometry group of $\calX_\KK(p,\kappa)$ is $\PGO_\K(p,\kappa).$

More precisely,

\begin{itemize}
\item $\Isom(\calX_\RR(p,\kappa))=\PO_\RR(p,\kappa)$,
\item 	$\Isom(\calX_\CC(p,\kappa))=\PO_\CC(p,\kappa)\rtimes \ZZ/2\ZZ$,
\item $\Isom(\calX_\H(p,\kappa))=\PO_\H(p,\kappa)$.
\end{itemize}
	
\end{thm}

\begin{proof}As explained above, this statement is not new in case $p=1$. So let us assume that $p\geq2$. Any isometry induces an automorphism of the associated polar space by Lemma~\ref{polar_aut}. By \cite[Theorem 8.6.II]{MR0470099}, since the dimension of $\mathcal{H}$ is at least 4, the group of all this automorphisms is  $\PGO_\K(p,\kappa).$\end{proof}

\begin{rem} The definition of $\PGO_\K(p,\kappa)$ in \cite[\S8.2.8]{MR0470099} is not exactly the same as the one here but it is easy to see that  they coincide. Let $f$ be some semilinear map associated to $\sigma\in\Aut(\C)$ such that $f^*Q=\alpha Q$ for some $\alpha\in\R^\ast$. Tits defined $\PGO_\C(p,\kappa)$ as a quotient of such semilinear maps. We claim that for such $f$, necessarily $\sigma\in\Aut_c(\C)$ thus one recovers the definition here. Let $x\in\mathcal{H}$ such that $Q(x)>0$. For any $\lambda\in\C$, $f^*Q(\lambda x)=\sigma^{-1}\left(|\sigma(\lambda)|^2 Q(f(x))\right)$. So,

$$|\sigma(\lambda)|^2Q(f(x))=\sigma\left(\alpha Q(x)\right)\sigma(|\lambda|^2).$$

For $\lambda=1$, we get that $Q(f(x))=\sigma\left(\alpha Q(x)\right)$ and thus $|\sigma(\lambda)|^2=\sigma(|\lambda|^2)$ which implies that $\sigma$ preserves $\R$ and thus $\sigma\in\Aut_c(\C)$.
\end{rem}
\subsection{Topological simplicity}\label{sec_top_simp}
Let us endow $\Isom(\calX_\KK(p,\kappa))$ with the \emph{topology of pointwise convergence}, that is the topology associated to the uniform structure given by the \'ecarts $(g,h)\mapsto d(gx,hx)$ for any $x\in\calX_\KK(p,\kappa)$. When $\kappa\leq\infty$, this topology on $\Isom(\calX_\K(p,\kappa))$ is Polish \cite[\S 9.B]{MR1321597}.
The group $\PO_\K(p,\infty)$ is a closed subgroup and thus a non-locally compact Polish group. 

\begin{lem}The group $\PO_\K(p,\kappa)$ is a closed subgroup of $\Isom\left(\calX_\K(p,\kappa)\right)$.
\end{lem}

\begin{proof} In the real and quaternionic cases, there is nothing to prove since the two groups coincide. 

In the complex case, the symmetric space $\calX_\C(p,\kappa)$ is Hermitian and there is a Kähler form $\omega$. For three points $x,y,z$, one can define 

$$c(x,y,z)=\int_{\Delta_{(x,y,z)}}\omega.$$

This is a continuous $\PO_\CC(p,\kappa)$-invariant cocycle \cite[\S5.2]{duchesne2018boundary}. If $g$ is an anti-holomorphic isometry (i.e. a $\sigma$-semilinear isometry where $\sigma$ is the complex conjugation) then $c(gx,gy,gz)=-c(x,y,z)$. Fix $x,y,z$ such that $c(x,y,z)\neq0$. Since the map $c\colon \calX_\C(p,\kappa)^3\to\R$ is continuous then $\PO_\CC(p,\kappa)$ is the preimage of $\{1\}$ under the continuous map $g\mapsto \frac{c(gx,gy,gz)}{c(x,y,z)}$ and thus closed.\end{proof}

In the remaining of this section we prove $\PO_\K(p,\kappa)$ is topologically simple for any cardinal $\kappa$. For a closed non-degenerate subspace $E$, $\mathcal{H}=E\oplus E^\bot$ and the orthogonal group of the restriction of $Q$ on $E$ can be embedded in $\OO_\K(p,\kappa)$ by letting it act trivially on $E^\bot$. We denote by $\OO_\K(E)$ its image in $\OO_\K(p,\kappa)$ and by $\PO_\K(E)$ the corresponding subgroup of $\PO_\K(p,\kappa)$

\begin{prop}\label{OE} Let $g_1,\dots,g_n$ be a finite collection of elements in $\Ofr(p,\kappa)$ then there is a non-degenerate finite dimensional subspace $E\leq\mathcal{H}$ of index $p$ such  that $g_i\in\OO_\K(E)$ for all $i\leq n$.
\end{prop}

\begin{proof} We first prove the result for one element $g\in\Ofr(p,\kappa)$. By definition, there is an operator $A$ of finite rank such that $g=\Id+A$. Let $F$ be the image of $A$ and choose $E$ to be any finite dimensional non-degenerate subspace of $\mathcal{H}$ of index $p$ that contains $F$. Observe that any subspace of $\mathcal{H}$ is $g$-invariant if and only if it is $A$-invariant. Since for any $x\in E$, $Ax\in F\subset E$, $E$ is $g$-invariant. Since $E$ is non-degenerate, $\mathcal{H}=E\oplus E^\bot$ is a $g$-invariant splitting. We claim that the restriction of $A$ to $E^\bot$ is trivial. Actually, for $x\in E^\bot$, $Ax\in E^\bot$ (because $E^\bot$ is $g$-invariant) and $Ax\in F\subset E$. Thus $Ax=0$. So $g\in \OO_\K(E)$.

Now, for each $i$, we find $E_i$ non-degenerate finite dimensional subspace of index $p$ such that $g_i\in\OO_\K({E_i})$. The sum $E=\sum_i E_i$ is non-degenerate finite dimensional subspace of index $p$. Since $\OO_\K({E_i})\le\OO_\K(E)$, the result follows.
\end{proof}

There is no determinant for infinite dimensional operators in general but if $\K=\R$ or $\C$ and $g\in\OO_\K^\mathfrak{fr}(p,\kappa)$, one can find $E$ as in Lemma~\ref{OE} and define $\det(g)=\det(g_E)$ which does not depend on the choice of $E$. As in finite dimension, this defines a group homomorphism $\OO_\K^\mathfrak{fr}(p,\kappa)\to\{\pm1\}$ or $\mathbf{S}^1$. We define $\SO^\mathfrak{fr}_\K(p,\kappa)$ to be the subgroup of elements with determinant 1. In  the quaternionic case, there is also a notion of determinant \cite{MR0012273} but it  takes positive real values and thus it is constant to 1 on $\OO_\K(p,q)$ for $p,q\in\NN$. 

So, we use  $\SO^\mathfrak{fr}_\H(p,\kappa)=\OO^\mathfrak{fr}_\H(p,\kappa)$ to have a uniform notation. We denote by $\PSO_\K^{\mathfrak{fr}}(p,\kappa)$ the image of $\SO_\K^{\mathfrak{fr}}(p,\kappa)$ in $\PO_\K^{\mathfrak{fr}}(p,\kappa)$. Let us observe that this image is isomorphic to $\SO_\K^{\mathfrak{fr}}(p,\kappa)$ since non-trivial homotheties are never finite rank perturbations of the identity.\\

The real case is bit particular since the Lie group $\SO_\R(E)$ has two connected components. If $P$ is a definite positive subspace of dimension $p$ in $E$ and let $\pi_P$ be the orthogonal projection to $P$. The map 
$$\begin{array}{rcl}
\SO_\R(E)&\to&\{\pm1\}\\
g&\mapsto& \frac{\det\left((\pi_P\circ g)|_P\right)}{\left|\det\left((\pi_P\circ g)|_P\right)\right|}
\end{array}$$ 

is a group homomorphism and its kernel is exactly the connected component of the identity in $\SO_\R(E)$, which is called the \emph{orthochronous group}\footnote{The name comes from the Lorentz group in special relativity. The orthochronous subgroup is the subgroup that preserves the direction of time.} denoted $\SO_\R^o(E)$. This homomorphism does not depend on the choice of $P$ and actually extends to the whole of $\OO_\R(p,\kappa)$. 

\begin{lem}The orientation map 

$$\begin{array}{rcl}
\mathop{or}\colon\OO_\R(p,\kappa)&\to&\{\pm1\}\\
g&\mapsto& \frac{\det\left((\pi_P\circ g)|_P\right)}{\left|\det\left((\pi_P\circ g)|_P\right)\right|}
\end{array}$$ 
is a well-defined surjective homomorphism.
\end{lem}
\begin{proof} Let us denote by $\mathcal{B}_+=\{x=(x_1,\dots,x_p)\in \mathcal{H}^p,\ \Span(x)\in\calX_\R(p,\kappa)\}$. Let $\bigwedge^p\mathcal{H}$ the $p$-th exterior power of $\mathcal{H}$ and for $g\in\GL(\mathcal{H})$, we denote by $\wedge^pg$ the linear operator defined by $\wedge^pg(x_1\wedge\dots\wedge x_p=g(x_1)\wedge\dots\wedge g(x_p)$. The symmetric form $\varphi$ on $\mathcal{H}$ induces a symmetric form on $\bigwedge^p\mathcal{H}$   denoted $\varphi$ as well and defined by 
$$ \varphi(x_1\wedge\dots\wedge x_p,y_1\wedge\dots\wedge y_p)=\det\left(\varphi(x_i,y_j)\right).$$ 

There is an $\OO_\R(p,\kappa)$-equivariant continuous map $i\colon \mathcal{B}_+\to\bigwedge^p\mathcal{H}$ defined by $i(x)=x_1\wedge\dots\wedge x_p$ where $x=(x_1,\dots,x_p)$. Let us observe that $i(x)$ and $i(y)$ are proportional if and only if $\Span(x)=\Span(y)$ and they are moreover positively proportional if the bases $x$ and $y$ are in the same orientation class.

For $x,y\in\mathcal{B}_+$, $\varphi(i(x),i(y))\neq0$ since there exist orthonormal bases $x’,y’$ of $\Span(x)$ and $\Span(y)$ such that $\varphi(x’_i,y’_i)=\cosh(\lambda_i)$ where $(\lambda_i)$ is the collection of hyperbolic principal angles between $\Span(x)$ and $\Span(y)$ and  $\varphi(x’_i,y’_j)=0$ for $i\neq j$. So, $\varphi(i(x’),i(y’))=\prod\cosh(\lambda_i)\neq0$. Moreover for $x,y,z\in\mathcal{B}_+$ if $\varphi(i(x),i(y))>0$ and $\varphi(i(x),i(z))>0$ then $\varphi(i(y),i(z))>0$. This follows from the fact that there is a continuous path $t\mapsto x_t$ from $[0,1]$ to $\mathcal{B}_+$ such that $x_0=x$ and $x_1$ is a basis of $\Span(y)$ in the orientation class of $y$ (using the exponential matrix that appears in \cite[Proposition 3.5]{MR3044451} which maps $\Span(y)$ to $\Span(x)$ and multiplying the hyperbolic angles $\lambda_i$ by $(1-t)$). In particular, one can define an equivalence relation $\sim$ on $\mathcal{B}_+$ such that $x\sim y\iff \varphi(i(x),i(z))>0$. This relation is invariant under the induced action of $\OO_\R(p,\kappa)$ (i.e. $gx\sim gy\iff x\sim y$). In particular, an element $g\in\OO_\R(p,\kappa)$ preserves each of the two equivalent classes or permute them. It preserves these two classes if and only if for any $x\in\mathcal{B}_+$, $\varphi(i(gx),x)>0$. 

If $x$ is a basis of $P$ then $\varphi(i(gx),x)$ and $\det\left((\pi_P\circ g)|_P\right)$ have the same sign. This yields the homomorphism $\OO_\R(p,\kappa)\to\ZZ/2\ZZ$.  The surjectivity is obtained by choosing an element of $\OO_\R(P)\setminus\SO_\R(P)$ and extending it trivially on $P^\bot$. 
\end{proof}

We denote the kernel of $\mathop{or}$ by $\OO^o_\R(p,\kappa)$ and called it the \emph{orthochronous group} as well. 
\begin{rem}Let us observe that if $p$ is odd then $-\Id\notin\OO^o_\R(p,\kappa)$ and in this case,  $\PO_\R^o(p,\kappa)=\PO_\R(p,\kappa)$. On the contrary, if $p$ is even then $-\Id\in\OO^o_\R(p,\kappa)$ and the orientation map $\mathop{or}$ defines a non-trivial homomorphism on the quotient group $\PO_\R(p,\kappa)$. Thus $\PO^o_\R(p,\kappa)$ is a strict non-trivial normal subgroup of $\PO_\R(p,\kappa)$.
\end{rem}
%

\begin{lem}\label{frdense} For any finite collection $\{x_i\}_{i\in\{1,\dots,n\}}$ of points in $ \mathcal{X}_\K(p,\kappa)$ and $g\in\PO_\K(p,\kappa)$, there is $g_0\in \PSO_\K^{\mathfrak{fr}}(p,\kappa)$ such that $gx_i=g_0x_i$ for any $i\in\{1,\dots,n\}$.\end{lem}

\begin{proof} Let $E$ be the minimal non-degenerate subspace of $\mathcal{H}$ such that $x_i\subseteq E$ for all $i\in[1,n]$ \cite[Lemma 3.11]{duchesne2018boundary}. Let us set $F=E+g(E)$ which is a non-degenerate subspace and thus $\mathcal{H}=F\oplus F^\bot$. By Witt theorem, there is $h\in\PO(F)$ such that $h^{-1}gE=E$. In particular $h\in \PO_\K^{\mathfrak{fr}}(p,\kappa)$. The element $h^{-1}g$ preserves $E$ and thus its orthogonal $E^\bot$. Let us define $f$ to be the restriction $h^{-1}g$ on $E$ and being trivial on $E^\bot$. Let $g_0=hf\in\PO_\K^{\mathfrak{fr}}(p,\kappa)$. Since for all $i\in[1,n]$ $x_i \subset E$, $g_0x_i=h(h^{-1}g)x_i=gx_i$.
In case $\K=\R$ or $\C$, one can moreover choose a (necessarily negative) line $L$ in $F^\bot$ and modify $g_0$ to act by $\det(g_0)^{-1}$ on this line and being trivial on $\left(F+L\right)^\bot$. In this way, one has moreover $g_0\in\SO_\K^\mathfrak{fr}(p,\kappa)$.
\end{proof}

\begin{prop}\label{dense} The group $\PSOfr(p,\kappa)$ is simple and dense in $\PO_\K(p,\kappa)$ when $\K=\C$ or $\H$. In the real case, $\PSO^{\mathfrak{fr}}_\R(p,\kappa)\cap\PO^o_\R(p,\kappa)$ is simple and dense in $\PO^o_\R(p,\kappa)$\end{prop}

\begin{proof} It is well known that when $q$ is finite, the group $\PSO_\KK(p,q)$ or $\PSO_\R^o(p,q)$, is a connected simple Lie group with trivial center and thus is abstractly simple. Let $g\in\PSOfr(p,\kappa)$ (assumed to be moreover orthochronous in the real case), there is $E$ such that $g\in\PSO_\K(E)$. If $h\in\PSOfr(p,\kappa)$ is another element (also assumed to be orthochronous in the real case), we may moreover enlarge $E$ such that $g,h\in\PSO_\K(E)$. If $g$ is not trivial then the normal subgroup of $\PSO_\K(E)$ (respecticvely $\PSO_\R^o(E)$) generated by $g$ is  $\PSO_\K(E)$ (respecticvely $\PSO_\R^o(E)$) itself  and thus contains $h$. So the simplicity  statement follows.

Density of $\PSO^{\mathfrak{fr}}_\K(p,\kappa)$ in $\PO_\K(p,\kappa)$ is  a straightforward corollary of Lemma~\ref{frdense}. In case $\K=\R$ and $g$ is orthochronous then $g_0$ (from Lemma~\ref{frdense}) is automatically orthochronous as well.
\end{proof}

\begin{rem}For any ideal $\mathcal{I}$ of the space of bounded operators (like the finite rank operators, the compact operators or the $q$-th Schatten class operators for $q\geq1$) one can construct a normal subgroup of elements of $\OO_\K(p,\kappa)$ that are perturbations of the identity where the perturbation lies in $\mathcal{I}$, i.e. elements of the form $I+A$ where $A\in\mathcal{I}$. So, $\OO_\K(p,\kappa)$ has a lot of normal subgroups and $\PO_\K(p,\kappa)$ is not simple. 
\end{rem}

For a subgroup $G'\leq G$, we denote by $\mathcal{Z}_G(G')$ its centralizer in $G$. A topological group is said to be \emph{topologically simple} if any non-trivial normal subgroup is dense. In order to show that $\PO_\K(p,\kappa)$ is topologically simple, let us show two easy facts from group theory.

\begin{lem}\label{normalsb} Let $G$ be a group and $N$ a normal simple subgroup. If $H\le G$ is a normal subgroup, either $N\leq H$ or $H\le \mathcal{Z}_G(N)$.
\end{lem}

\begin{proof} Let $n\in N$ and $h\in H$ then $[n,h]\in N\cap H$. So if $H\nleq \mathcal{Z}_G(N)$ then there are $n$ and $h$ such that $[n,h]\neq e$ thus $N\cap H$ is non trivial normal subgroup of $N$ and thus $N=N\cap H\leq H$.
\end{proof}

\begin{prop}\label{topsimp} Let $G$ be a Hausdorff topological group with trivial center. If there is a dense normal simple subgroup $N$ then $G$ is topologically simple.
\end{prop}

\begin{proof}Let $H$ be a non trivial normal subgroup. Assume $H\leq\mathcal{Z}_G(N)$. Let $h\in H$, the map $g\mapsto[g,h]$ from $G$ to $G$, is continuous and constant on $N$. Since $N$ is dense and $G$ is Hausdorff, this map is constant on $G$ and $h$ is in the center of $G$. Since $H$ is non-trivial and $G$ has trivial center, this is a contradiction. So by Lemma~\ref{normalsb}, $N\leq H$ and $H$ is dense.
\end{proof} 

 \begin{thm}\label{thmtopsimpk} Let $\kappa$ be some infinite cardinal. The topological groups $\PO_\C(p,\kappa),\PO_\H(p,\kappa)$ and $\PO_\K^o(p,\kappa)$ are topologically simple.
 \end{thm}
 In case $\kappa=\infty$, one recovers Theorem \ref{thmtopsimp}.
\begin{proof}Since $\PO_\K(p,\kappa)$ or $\PO_\R^o(p,\kappa)$ has trivial center, the theorem follows from Propositions \ref{dense} and \ref{topsimp}.\end{proof}


%

\section{Actions on infinite dimensional symmetric spaces}
\begin{defn} A subgroup $G$ of $\Isom(\calX_\K(p,\kappa))$ is \emph{geometrically dense} if $G$ has no fixed point at infinity nor invariant strict totally geodesic subspace. A representation $\rho\colon G\to\Isom(\calX_\K(p,\kappa))$ is geometrically dense if its image is so.
\end{defn}

\begin{rem} Our definition is different from the one in \cite{MR2574741}. There, totally geodesic subspaces are replaced by convex subspaces because they consider \cat spaces instead of symmetric spaces. For symmetric spaces of non-compact type and rank 1, the two definitions do not coincide but for finite dimensional irreducible higher rank symmetric spaces, both definitions are equivalent to Zariski-density.

Our definition coincides with the one in \cite{duchesne2018boundary} and with \emph{geometric Zariski density} in \cite[\S5]{MR3263898}
\end{rem} 

In \cite[Lemma 4.2]{MR2574740}, it is proved that for any group $G\leq\Isom(\mathcal{X})$ of a \cat space $\mathcal{X}$, the boundary of the convex closure any $G$-orbit does not depend on the choice of the orbit. Since the normalizer $\mathcal{N}(G)$ of $G$ permutes the $G$-orbits, this yields a subspace $\Delta G\subset\partial \mathcal{X}$, namely the convex closure of  any orbit, which is $\mathcal{N}(G)$-invariant.

In \cite{MR1704987} Leeb showed that the \emph{geometric dimension} of a {\upshape CAT($1$)}\xspace space $\mathcal{X}$, a notion he introduced, can be computed as the supremum of the topological dimension of compact subsets of $\mathcal{X}$. A \cat space $\mathcal{X}$ has \emph{telescopic dimension} at most $k\in\NN$ if any asymptotic cone of $\mathcal{X}$ has geometric dimension at most $k$. If such a finite $k$ exists then $\mathcal{X}$ is said to have \emph{finite telescopic dimension} and in this case its telescopic dimension is defined to be the minimal such $k$.

 A main feature of the spaces $\mathcal{X}_\KK(p,\kappa)$ is their finite telescopic dimension, which is exactly $p$. Thus it coincides with the rank of $\mathcal{X}_\KK(p,\kappa)$ which is the dimension of maximal flat subspaces. We refer to \cite{MR2558883} for details about this dimension and \cite{MR3044451} for a computation of this telescopic dimension for  $\mathcal{X}_\KK(p,\kappa)$.

A consequence of the finiteness of this telescopic dimension is the following fact.

\begin{thm}[{\cite[Theorem 1.1]{MR2558883}}]\label{cl}Let $\mathcal{X}$ be a complete \cat space of finite telescopic dimension and
${\left(\mathcal{X}_\alpha\right)}_{\alpha\in A}$ be a filtering family of closed convex subspaces. Then either the intersection
$\cap_{\alpha\in A} \calX_\alpha$ is non-empty, or the intersection of the visual boundaries $\cap_{\alpha\in A} \partial \calX_\alpha$  is a 
non-empty subset of $\partial \mathcal{X}$ of intrinsic radius at most $\pi/2$.
\end{thm}

Moreover the boundary of a \cat space of telescopic dimension has finite geometric dimension \cite[Proposition 2.1]{MR2558883} and this allows us to use the following fixed point statement.

\begin{prop}[{\cite[Proposition 1.4]{MR2180749}}]\label{cl2}Let $\mathcal{X}$ be a {\upshape CAT($1$)}\xspace space of finite dimension and of  intrinsic radius at most $\pi/2$. Then $\mathcal{X}$ has a circumcenter which is fixed by every isometry of $\mathcal{X}$.
\end{prop}

\begin{lem}\label{lem:min} Let $G$ be a subgroup of $\Isom(\mathcal{X}_\KK(p,\kappa))$ whose normalizer $\mathcal{N}(G)$ in $\Isom(\mathcal{X}_\KK(p,\kappa))$ has no fixed point at infinity. Then $G$ has a unique minimal  invariant closed convex subset and a unique minimal invariant totally geodesic subspace.
\end{lem}

\begin{proof} Let $\mathcal{C}$ be the collection of either $G$-invariant closed convex subsets or $G$-invariant totally geodesic subspsaces. This collection is non-empty because it contains $\mathcal{X}_\KK(p,\kappa)$. Moreover, any intersection of elements of $\mathcal{C}$ is either empty or an element of $\mathcal{C}$. If $\mathcal{C}$ has no minimal element (for inclusion) then one can find a sequence $(x_i)$ such that $\overline{\Conv(Gx_{i+1})}\subset \overline{\Conv(Gx_{i})}$. So $\Delta G$ is non-empty set of intrinsic radius at most $\pi/2$.  Since $\Delta G$ is $\mathcal{N}(G)$-invariant, $\mathcal{N}(G)$ has a fixed point. This is a contradiction and this implies that $\mathcal{C}$ has a minimal element.

Let us prove first the uniqueness of a minimal closed convex $G$-invariant subspace. 
Let $\mathcal{Y}$ be the union of all  minimal closed convex $G$-invariant subspaces. Let $\mathcal{X}$ be such a minimal  closed convex $G$-invariant subspace. By \cite[Remark 39.(1)]{MR2219304}, $\mathcal{Y}$ splits isometrically as $\mathcal{Y}\simeq\mathcal{X}\times\mathcal{T}$ where the action of $G$ on $\mathcal{Y}$ is diagonal, being trivial on $\mathcal{T}$. Assume that $\mathcal{T}$ is not reduced to a point, that is there are two minimal subspaces $\mathcal{X}$ and $\mathcal{X}'$. Let $x\in \mathcal{X}$ and $x'$ its projection on $\mathcal{X}'$. For any $g\in G$, the convex hull of the segments $\left[x,x'\right]$ and $\left[\gamma x,\gamma x'\right]$ is a flat strip. The four vertices and their convex hull lie in some totally geodesic subspace of finite dimension. This flat strip is thus contained  in some flat totally geodesic subspace of $\mathcal{X}_\KK(p,\kappa)$. So, if $\ell$ is the geodesic line through $x$ and $x'$, then $\gamma\ell$ is parallel to $\ell$ and the two points at infinity of $\ell$ are fixed by $g$. Thus we have a contradiction and there is a unique minimal closed convex invariant subspace $\mathcal{X}$.

If there were two minimal invariant totally geodesic subspaces, each one would have a minimal closed convex invariant subspace by applying the argument at the beginning of the proof to each of them. Thus, we know there is a unique minimal invariant totally geodesic subspace. \end{proof}

%

In the fifties, Tits conjectured that convex subcomplexs of spherical buildings are either buildings themselves or they have a center, i.e. a fixed point for all automorphisms preserving the subcomplex. The interest for this conjecture was renewed in relation to Serre's complete reducibility \cite[\S 2.4]{MR2167207}. It was proved first in the classical cases \cite{MR2228217} and extended later in the exceptional cases. We recommend  \cite{ediss11005} for details about this conjecture.

\begin{thm}[Solution of the Center Conjecture]\label{centerconjecture} Let $\Delta$ be some spherical building of type $B_p$. If $C$ is a convex subcomplex of $\Delta$, either $\Delta$ has a center, or any simplex in $\Delta$ has an opposite simplex in $\Delta$.
\end{thm}

\begin{cor}\label{corcc} Let $G$ be a subgroup of $\Isom\left(\calX_\K(p,\kappa)\right)$ with a non-empty fixed point set at infinity $Y$. Either $Y$ is a sub-building of $\partial\calX_\K(p,\kappa)$ or the normalizer of $G$ in $\Isom\left(\calX_\K(p,\kappa)\right)$ has a fixed point at infinity.
\end{cor}

\begin{proof} The group $\Isom\left(\calX_\K(p,\kappa)\right)$ acts by preserving type automorphisms on the spherical building $\partial\calX_\K(p,\kappa)$ (Lemma~\ref{polar_aut}), which is of type $B_p$. The set of fixed points in $\partial\calX_\K(p,\kappa)$ of some isometry $g\in\Isom\left(\calX_\K(p,\kappa)\right)$ is convex (because it induces an isometry for the Tits metric on  $\partial\calX_\K(p,\kappa)p$) and it is a subcomplex because if some $\xi\in \partial\calX_\K(p,\kappa)$ is fixed by $g$ the closure of the smallest facet that contains $\xi$ is pointwise fixed.

By Theorem~\ref{centerconjecture}, The fixed point set $Y$ of $G$ is either a sub-building or there is a point $\xi\in Y$ which invariant under all typer preserving automorphisms that stabilize $Y$. Since the normalizer of $G$ stabilizes $Y$, we have the result.
\end{proof}


The following proposition is similar to \cite[\S4.C]{MR2574740}
\begin{prop}\label{normal} Let $G,N$ be subgroups of $\Isom(\calX_\K(p,\kappa))$ such that $N$ is a normal subgroup of $G$. The subgroup $G$ is geometrically dense if and only if $N$ is so.
\end{prop}

\begin{proof}It is straightforward that if $N$ is geometrically dense so is $G$. So, let us assume that $G$ is geometrically dense.  By Lemma~\ref{lem:min}, since $G$ has no fixed point at infinity, $N$ has a unique minimal invariant totally geodesic subspace $\mathcal{Y}$. Its uniqueness implies that it is $G$-invariant as well. Since $G$ is geometrically dense, $\mathcal{Y}=\calX_\K(p,\kappa)$.

Now, If $N$ has fixed points at infinity then by Corollary~\ref{corcc} either $N$ has two opposite fixed points or the set of $N$-fixed points at infinity has a center and $G$ has a fixed point at infinity. The last case is a contradiction.

If $N$ has opposite fixed points, there are two isotropic subspaces $E_-,E_+$ of the same dimension $k\leq p$ such that $E_-+E_+$ is non-degenerate. Let $\mathcal{Y}$ be the set of elements $E$ of $\calX_\K(p,\kappa)$ such that $E\cap (E_++E_-)$ has dimension $k$. The subspace $\mathcal{Y}$ is a strict $N$-invariant totally geodesic subspace and once again we have a contradiction.

So $N$ has no fixed point at infinity and thus is geometrically dense. 
\end{proof}

%
 
 \section{Representations of finite dimensional simple Lie groups of rank at least 2}
 Let $G$ be a connected semisimple Lie group with trivial center and no compact factor. Then $G$ is the connected component to the isometry group of a symmetric space of non-compact $\mathcal{X}_G$ which is of the form $G/K$ where $K$ is a maximal compact subgroup.
  
 \begin{prop}\label{toteq}Let $G$ be a connected simple non-compact Lie group with trivial center. Let $\rho\colon G\to\Isom(\calX_\K(p,\kappa))$ be a continuous representation without fixed point in $\mathcal{X}_{\K}(p,\kappa)$ nor in $\partial \calX_\K(p,\kappa)$.\\
 If the real rank of $G$ is at least 2 then there is a $G$-equivariant totally geodesic isometric embedding of the symmetric subspace $\mathcal{X}_G$ in $\calX_\K(p,\kappa)$. 
 \end{prop}
 
 \begin{proof} Let $\Gamma$ be a cocompact lattice in $G$ (which exists, see \cite{MR3307755} for example). Moreover, up to consider a finite index subgroup, we may assume that $\Gamma$ is torsion free. Let us prove first that there is a $\Gamma$-invariant totally geodesic isometric embedding $\mathcal{X}_G\to\calX_\K(p,\kappa)$. \\

 Since $G$ has no fixed point at infinity then it has a unique minimal closed invariant subspace $\calX$ (Lemma~\ref{lem:min}). This space $\calX$ has finite telescopic dimension. It has a de Rham decomposition \cite[Proposition 6.1]{MR2558883} and the Euclidean factor has to be trivial since $G$ has property (T) and thus a fixed point on this Euclidean factor.
 
 We claim that there is a $\Gamma$-equivariant harmonic map $\calX_G\to\calX$ in the sense of Korevaar and Schoen for metric spaces. The existence of such harmonic map is provided by \cite[Theorem 3.1]{MR3343349} if we know that $\Gamma$ has no fixed point in $\partial\calX$.  

If $\Gamma$ has a fixed point at infinity $\xi\in\partial\calX$, one can define a map $f\colon \calX\to \RR$ given by the formula
$$f(x)=\int_{G/\Gamma}\beta_{g\xi}(x,x_0)d\mu(g\Gamma)$$
where $\mu$ is the $G$-invariant probability measure on $G/\Gamma$ and $x\mapsto\beta_{\eta}(x,x_0)$ is the Busemann function associated to the point $\eta\in\partial \calX_\K(p,\kappa)$ that vanishes at some base point $x_0$.

The function $f$ is convex, 1-Lipschitz and $G$-almost invariant, that is $x\mapsto f(x)-f(gx)$ is constant for any $g\in G$.  If $f$ has no minimum then $G$ has a fixed point at infinity (the center at infinity associated to the invariant filteringfamily of sub-level sets of $f$ \cite[Theorem 1.1]{MR2558883}), that is a contradiction and if $f$ has a minimum then $f$ is constant since $\calX$ is a minimal $G$-invariant convex subspace. Thanks to \cite[Proposition 4.8]{MR2558883}, $\calX$ splits as $\calX'\times\R$  and we have a contradiction with the vanishing of the Euclidean factor of $\calX$. 

So there is a $\Gamma$-equivariant harmonic map $h\colon\calX_G\to\calX$. Assume this map is constant. Then the image is a $\Gamma$-fixed point $x\in\calX_\K(p,\kappa)$ and by continuity of the orbit map $g\Gamma \mapsto gx$, the $G$-orbit of $x$ is compact and thus bounded. In particular, $G$ has a fixed point, that is a contradiction. 

Since the metric projection $\pi_\calX\colon\calX_\K(p,\kappa)\to\calX$ is 1-Lipschitz and $\Gamma$-equivariant, the map $u\mapsto\pi_\calX\circ u$ does not increase the energy of $\Gamma$-equivariant maps. Thus $h$ is harmonic as well as map from $\calX_G$ to $\calX_\K(p,\kappa)$. We can now use \cite[Proposition 4.1]{MR3343349} to show that $h$ is a smooth map and conclude as in \cite[Theorem 1.2]{MR3343349} to prove that $h$ is actually a totally geodesic embedding.

So we know that there exist totally geodesic embeddings $\mathcal{X}_G\to\calX_\K(p,\kappa)$ that are $\Gamma$-equivariant.  Let us denote by $\mathcal{Y}$  some image of a $\Gamma$-equivariant totally geodesic embedding of $\mathcal{X}_G$ in $\calX$ and let us introduce the function $\varphi\colon \calX\to \RR^+$ given by 


$$\varphi(x)=\int_{G/\Gamma}d(x,g\mathcal{Y})^2d\mu(g\Gamma).$$

The function is well defined because $\mu$ is $G$-invariant and $g\mapsto d(x,g\mathcal{Y})$ is bounded by cocompactness. The function $\varphi$ is $G$-invariant, continuous and convex. If it has no minimum then $G$ has a fixed point at infinity (by the same argument about the filteringfamily of sub-level sets)  and otherwise the minimum is realized on the whole of $\calX$ since this space is minimal among closed convex $G$-invariant subspaces. Let $x,y\in\calX$ and $m$ their midpoint. Using the fact that $\varphi$ is constant on $\calX$ and $x\mapsto d(x,g\mathcal{Y})$ is convex, one has

\begin{align*}
0&=\varphi(m)-\frac{\varphi(x) +\varphi(y)}{2}\\
&\leq\int_{G/\Gamma}\left(\frac{d(x,g\mathcal{Y})+d(y,g\mathcal{Y})}{2}\right)^2-\left(\frac{d(x,g\mathcal{Y})^2+d(y,g\mathcal{Y}_\Gamma)^2}{2}\right)d\mu(g\Gamma)\\
&\leq -\frac{1}{4}\int_{G/\Gamma}\left(d(x,g\mathcal{Y})-d(y,g\mathcal{Y})\right)^2d\mu(g\Gamma)
\end{align*}
and thus  for all $g\in G$, $d(x,g\mathcal{Y})=d(y,g\mathcal{Y})$. By taking $x\in \mathcal{Y}$ and $g$ to be the identity, one has that $\mathcal{X}=\mathcal{Y}$. So, $\mathcal{X}$ is a $G$-invariant totally geodesic subspace isometric to $\mathcal{X}_G$. Let $h\colon \mathcal{X}_G\to\mathcal{X}$ be a $\Gamma$-equivariant totally geodesic embedding. Let $m\colon G\to G$ be the continuous group homomorphism defined by $m(g)=h^{-1}\circ \rho(g)\circ h$. Since the restriction of $m$ is the identity on $\Gamma$, $m$ is the identity. This means that $h$ is $G$-equivariant. 
\end{proof} 

\begin{rem}The strategy of the first part of the proof of Proposition~\ref{toteq} is very close to the one of \cite[Theorem 2.4]{MR2574740} but this theorem does not apply directly because the space $\mathcal{X}_\K(p,\kappa)$ is not proper. It could be deduce from \cite[Theorem 1]{MR3323639} but we prefer to give a proof without the vocabulary of IRS, which is useless here. The IRS related to this proof is the one supported on the conjugacy class of $\Gamma$. \end{rem}
  
  The following theorem is the extension of Theorem~\ref{splitting} to any cardinal $\kappa$. 
   \begin{thm}\label{splittingk}Let $G$ be a connected simple non-compact Lie group with trivial center. Let $G\to\PO_\K(p,\kappa)$ be a continuous representation without totally isotropic invariant subspace.\\
 If the real rank of $G$ is at least 2 then the underlying Hilbert space $\mathcal{H}$ splits orthogonally as $E_1\oplus\dots\oplus E_k\oplus \mathcal{K}$ where each $E_i$ is finite dimensional, non-degenerate, $G$-invariant and the induced representation on $E_i$ is irreducible. The induced representation on $\mathcal{K}$ is unitary.
 \end{thm} 
 
 \begin{proof}If there is a fixed point in $\calX_\K(p,\kappa)$, then there is a definite positive subspace $E$ of dimension $p$ which is $G$-invariant. Let $\mathcal{K}$ be the orthogonal of $E$. One has $\mathcal{H}=E\oplus\mathcal{K}$ and $\mathcal{K}$ is negative definite. So the induced representation on $\mathcal{K}$ is unitary. This implies that the representation on $\mathcal{H}$ is unitary for the scalar product $\langle\ ,\ \rangle|_E-\langle\ ,\ \rangle|_{\mathcal{K}}$.
 
 Assume there is no fixed point in $\calX_\K(p,\kappa)$. Since a fixed point at infinity would yield an invariant totally isotropic flag, we know there is no fixed points at infinity.  So, by Proposition \ref{toteq}, there is a $G$-equivariant totally geodesic embedding of $\mathcal{X}_G$ in $\calX_\K(p,\kappa)$ and  we denote by $\calX_G$ its image as well. By \cite[Lemma 3.11]{duchesne2018boundary}, there is a unique minimal finite dimensional non-degenerate subspace $E_0\leq\mathcal{H}$  such that for any $x\in\mathcal{X}_G$, $x\leq E_0$. By uniqueness, this linear subspace $E_0$ is $G$-invariant. If $E\leq E_0$ is $G$-invariant then its kernel would be $G$-invariant and there would be a totally isotropic invariant space. So, any $G$-invariant subspace $E\leq E_0$ is non-degenerate and in particular, $\mathcal{H}$ splits orthogonally as $E\oplus E^\bot$. Let us choose such a minimal $E$. In particular the induced representation on $E$ is irreducible. The signature of the restriction of the quadratic form on $E^\bot$ is $(s,\kappa)$ for some $s<r$. While $s>0$, one can repeat the argument and an induction on the index of the quadratic form gives the result.
 \end{proof}

\section{Representations of $\PO_\K(p,\kappa)$ for $p\geq2$.}

The goal of this section is to prove Theorem~\ref{quad} that describes geometrically dense continuous representations $\PO_\K(p,\infty)\to\PO_{\mathbf{L}}(q,\infty)$ in the generality of infinite cardinals $\kappa$ and $\lambda$.
   
   \begin{thm}\label{quadk} Let $p,q\in\NN$ with $p\geq2$ and $\K,\mathbf{L}\in\{\R,\C,\H\}$, If $\rho\colon \PO_\K(p,\kappa)\to\PO_{\mathbf{L}}(q,\lambda)$ is a continuous geometrically dense representation then $\K=\mathbf{L}$, $q=p$, $\kappa=\lambda$ and $\rho$ is induced by an isomorphism of quadratic spaces.  \end{thm}

  \begin{proof} Let $\rho\colon\PO_\K(p,\kappa)\to\PO_{\mathbf{L}}(q,\lambda)$ be a geometrically dense representation. We denote the respective underlying Hilbert spaces $\mathcal{H}$ and $\mathcal{H}'$ with quadratic forms $Q$ and $Q'$. Let $\mathcal{E}$ be the collection of non-degenerate finite dimensional subspaces of $\mathcal{H}$ of index $p$. 		Let $E\in\mathcal{E}$ and let us denote $G_E=\PSO_\K(E)$ or $\PSO^o_\R(E)$ in the real case. First, we prove that for $E$ of dimension large enough, $G_E$ has no fixed points in $\calX_{\mathbf{L}}(p,\kappa)$. For the sake of a contradiction, let us assume that for any $E\in \mathcal{E}$, $G_E$ has a non-empty set of fixed points $\mathcal{Y}_E\subset\calX_{\mathbf{L}}(q,\lambda)$. For $E\leq E'$, $\mathcal{Y}_{E'}\subseteq\mathcal{Y}_E$ and thus this is a filtering family of totally geodesic subspaces. If $\cap_{E\in\mathcal{E}}\mathcal{Y}_E$ is non-empty them $\PSO_\K^\mathfrak{fr}(p,\kappa)$ has fixed points but since $\PSO_\K^\mathfrak{fr}(p,\kappa)$ is geometrically dense by Proposition~\ref{normal}, we have a contradiction. If the intersection $\cap_{E\in\mathcal{E}}\mathcal{Y}_E$ is empty, then by Theorem~\ref{cl} and Proposition~\ref{cl2}, there is a $\PSO_\K^\mathfrak{fr}(p,\kappa)$-fixed point at infinity contradicting its geometric density.
  As a conlusion, we know there is some $E\in\mathcal{E}$ such that $\mathcal{Y}_E=\emptyset$ and the same holds for any $E'\geq E$.

   Now, let us show that $G_E$ has no fixed points in $\partial\calX_{\mathbf{L}}(q,\lambda)$. Assume for the sake of a contradiction that there are fixed points at infinity. Let us recall that one can associate a flag of totally isotropic subspaces of $\mathcal{H}'$ to any point at infinity. If $G_E$ stabilizes some totally isotropic subspace $F$ then one get a continuous representation $G_E\to\PGL(F)$. Since $G_E$ is simple and $F$ has dimension at most $q$, we know that as soon $\dim(G_E)>\dim(\PGL(F))$, this representation is trivial (because a continuous representation of a (finite dimensional) Lie group to another Lie group is automatically smooth) and thus $G_E$ fixes all lines in $F$. In particular, the fixed point set $(\PP\mathcal{H}’)^{G_E}$ of $G_E$ in the projective space $\PP\mathcal{H}’$ is non-empty. If $L_1, L_2$ are two $G_E$-invariant lines  in $\mathcal{H}’$, by the same argument, $G_E$ fixes all lines of their span. In particular, there is a closed linear subspace $H_E$ of $\mathcal{H}’$ such that $(\PP\mathcal{H}’)^{G_E}$ is the projective space of $H_E$. Then, considering the restriction of $Q'$ on $H_E$,  one has two possibilities :
     \begin{enumerate}[label=(\emph{\roman*})]
    \item $H_E$ has a non-trivial kernel or
    \item $H_E$ is non-degenerate.
    \end{enumerate}
In the first case, this means that there is some line which is orthogonal to all $G_E$-invariant invariant totally isotropic subspaces. In particular, all $G_E$-fixed points in $\partial\calX_{\mathbf{L}}(q,\lambda)$ lie in the ball of radius $\pi/2$ for the Tits metric around the vertex corresponding to such a line in the kernel.
    
    In the second case, one has $\mathcal{H}’=H_E\oplus H_E^\bot$. Observe that if $E’\geq E$ then $H_{E’}\leq H_E$ and thus $H_{E’}^\bot\geq H_E^\bot$.

Assume that the first possibility (i) holds for all $E\in\mathcal{E}$ then let us denote by $C_E$ the set of $G_E$-fixed points in $\partial\calX_{\mathbf{L}}(q,\lambda)$. This is a filtering family of closed convex subsets of  intrinsic at most $\pi/2$ and by \cite[Lemma 5.1]{MR2558883}, the intersection of all of them is not empty and thus one get a $\PSO_\K^\mathfrak{fr}(p,\kappa)$-fixed point in $\partial\calX_{\mathbf{L}}(q,\lambda)$ and together with Proposition~\ref{normal}, we get a contradiction with geometric density of the representation.

So we know that there is $E\in\mathcal{E}$ such that $H_E$ has trivial kernel and thus for any $E’\geq E$, the same holds. The closure of the union $\cup_{E\in\mathcal{E}}H_E^\bot$ is $\PSO_\K^\mathfrak{fr}(p,\kappa)$-invariant. Since geometric density implies irreducibility of the representation, this union is necessarily dense and thus the index of the restriction of the quadratic form on this union is $q$. In particular, one can find $E$ such that $H_E^\bot$ has index $q$ (since the index can be checked using finitely many points at a time). This yields a contradiction with the fact that $H_E$ has some isotropic line and thus index at least 1, which implies that the one of $H_E^\bot$ is at most $q-1$.

In conclusion, we know that $G_E$ has no fixed point in  $\calX_{\mathbf{L}}(q,\lambda)$ nor in $\partial\calX_{\mathbf{L}}(q,\lambda)$ for some $E\in\mathcal{E}$ and thus for all $E’\geq E$.
    For such $E\in \mathcal{E}$, there is a $G_E$-equivariant totally geodesic embedding $\varphi_E$ of the symmetric space $\mathcal{X}_E$ of $G_E$ in $\PO_{\mathbf{L}}(q,\lambda)$ by Proposition~\ref{toteq} since $p\geq2$. The image $\mathcal{Y}_E$ of the embedding $\varphi_E$ is a $G_E$-invariant totally geodesic subspace. It is minimal since the action on it is transitive. Since  $G_E$ has no fixed point at infinity, this totally geodesic subspace $\mathcal{Y}_E$ is unique (Lemma~\ref{lem:min}). In particular, this uniqueness implies that  if $E\leq E’$, then $\mathcal{Y}_E\subset \mathcal{Y}_{E'}$ and more precisely $\varphi_E$ is the restriction of $\varphi_{E'}$ on $\mathcal{X}_E$.
  
  The closure $\mathcal{Y}$ of the union of the subspaces $\mathcal{Y}_E$ (which is isometric to the space $\calX_\K(p,\kappa)$) is a $\PSO^{\mathfrak{fr}}_\K(p,\kappa)$-invariant totally geodesic subspace of $\calX_{\mathbf{L}}(q,\lambda)$. By Lemma~\ref{normal} and the fact that $\rho\left(\PO_\K(p,\kappa)\right)$ is geometrically dense,  $\mathcal{Y}$ is necessarily $\mathcal{X}_{\mathbf{L}}(q,\lambda)$. In particular, the two spaces $\calX_\K(p,\kappa)$ and  $\calX_{\mathbf{L}}(q,\lambda)$ have the same ranks, i.e. $p=q$. The induced isometry between the Tits buildings at infinity yields an isomorphism of polar spaces as in the proof of Theorem~\ref{isomgroup} and once again, relying on   \cite[Theorem 8.6.II]{MR0470099}, this isomorphism is given by some isomorphism of quadratic spaces,  this corresponds to a semilinear map and in particular $\K=\mathbf{L}$ and $\kappa=\lambda$.
 \end{proof} 
     
\bibliographystyle{../../Latex/Biblio/halpha}
\bibliography{../../Latex/Biblio/biblio}

\end{document}